\documentclass[12pt,reqno]{amsart}
\usepackage{amsthm, amscd, amsfonts, amssymb, graphicx, color}
\usepackage[bookmarksnumbered, colorlinks, plainpages,linkcolor=green,anchorcolor=blue,citecolor=blue,urlcolor=blue]{hyperref}
\usepackage{mathrsfs}
\usepackage{amssymb}
\usepackage{enumerate}
\usepackage{exscale}
\usepackage[all]{xy}
\usepackage{relsize}
\usepackage{pdfsync}
\numberwithin{equation}{section}

\makeatletter
\@namedef{subjclassname@2020}{%
  \textup{2020} Mathematics Subject Classification}
\makeatother

\usepackage[includemp,body={398pt,550pt},footskip=30pt,%
            marginparwidth=60pt,marginparsep=10pt]{geometry}

\newtheorem{theorem}{Theorem}[section]
\newtheorem{lemma}[theorem]{Lemma}
\newtheorem{proposition}[theorem]{Proposition}
\newtheorem{corollary}[theorem]{Corollary}
\theoremstyle{definition}

\numberwithin{equation}{section}

\newcommand{\D}{\mathcal{D}}

\newcommand{\h}{\mathscr{H}}
\newcommand{\B}{\mathfrak{B}}
\newcommand{\R}{\mathbb{R}}
\newcommand{\N}{\mathbb{N}}
\newcommand{\C}{\mathbb{C}}
\newcommand{\T}{\mathbb{T}}
\newcommand{\TT}{\mathbb{T}^{\infty}}
\newcommand{\DD}{\mathbb{D}}

\begin{document}
\title[Volterra operators between Hardy spaces]{Volterra operators between Hardy spaces of vector-valued Dirichlet series}
\date{April 7, 2024.
}

\author[J. Chen]{Jiale Chen}
\address{Jiale Chen, School of Mathematics and Statistics, Shaanxi Normal University, Xi'an 710119, China.}
\email{jialechen@snnu.edu.cn}

\thanks{This work was supported by the Fundamental Research Funds for the Central Universities (No. GK202207018) of China.}



\begin{abstract}
  \noindent Let $2\leq p<\infty$ and $X$ be a complex infinite-dimensional Banach space. It is proved that if $X$ is $p$-uniformly PL-convex, then there is no nontrivial bounded Volterra operator from the weak Hardy space $\mathscr{H}^{\text{weak}}_p(X)$ to the Hardy space $\mathscr{H}^+_p(X)$ of vector-valued Dirichlet series. To obtain this, a Littlewood--Paley inequality for Dirichlet series is established.
\end{abstract}
\maketitle


\section{Introduction}
\allowdisplaybreaks[4]

Throughout the paper, $X$ will always be a complex Banach space. Let $\D(X)$ be the space of Dirichlet series $\sum_{n\geq1}x_nn^{-s}$ with $\{x_n\}_{n\geq1}\subset X$ that converge at some point $s_0\in\C$, and let $\mathcal{P}(X)$ be the space of $X$-valued Dirichlet polynomials $\sum_{n=1}^Nx_nn^{-s}$. Recently, the study of functional-analytic aspects of the theory of (vector-valued) Dirichlet series has attracted great attention; see \cite{Bo15,CDS14,CDS18,CMS,CMST,DGMP,DP,DPS,DSS} and the references therein. In this note, we are going to investigate the properties of Volterra operators between some Hardy spaces of vector-valued Dirichlet series.

To clarify the definition of Hardy spaces of vector-valued Dirichlet series, we need the following notions. We denote by $\TT$ the infinite-dimensional complex polytorus carrying a normalized Haar measure $m_{\infty}$ that is induced by the normalized Lebesgue measure $m$ on the unit circle $\T\subset\C$. Given $1\leq p<\infty$, let $L_p(\TT,X)$ be the space of $p$-Bochner integrable functions $F:\TT\to X$ with respect to the Haar measure $m_{\infty}$. For any multi-index $\nu=(\nu_1,\dots,\nu_n,0,\dots)\in\mathbb{Z}^{(\infty)}$ (the set of eventually null sequences of integers) the $\nu$th Fourier coefficient $\widehat{F}(\nu)$ of $F\in L_1(\TT,X)$ is given by
$$\widehat{F}(\nu):=\int_{\TT}F(z)z^{-\nu}dm_{\infty}(z).$$
For $1\leq p<\infty$, the Hardy space $H_p(\TT,X)$ is defined as the closed subspace of $L_p(\TT,X)$ consisting of those functions $F$ with $\widehat{F}(\nu)=0$ for all $\nu\in\mathbb{Z}^{(\infty)}\setminus\mathbb{N}_0^{(\infty)}$, 
where $\mathbb{N}_0^{(\infty)}$ denotes the set of $\nu$'s in $\mathbb{Z}^{(\infty)}$ with $\nu_j\in\mathbb{N}_0:=\mathbb{N}\cup\{0\}$ for all $j$.

Let $\mathfrak{p}=\{\mathfrak{p}_j\}_{j\geq1}$ be the increasing sequence of prime numbers. Given $\nu\in\mathbb{Z}^{(\infty)}$, we write $\mathfrak{p}^{\nu}:=\mathfrak{p}_1^{\nu_1}\mathfrak{p}_2^{\nu_2}\cdots$. By the fundamental theorem of arithmetic, for any $n\in\mathbb{N}$  there exists a unique multi-index $\nu(n)\in\mathbb{N}_0^{(\infty)}$ such that $n=\mathfrak{p}^{\nu(n)}$. Recall that every $F\in L_1(\TT,X)$ is uniquely determined by its Fourier coefficients $\{\widehat{F}(\nu)\}_{\nu\in\mathbb{Z}^{(\infty)}}$. Consequently, for every $F\in H_1(\TT,X)$ we may define its Bohr transform $\B(F)$ as the following $X$-valued Dirichlet series:
$$\B(F)(s):=\sum_{n=1}^{\infty}\widehat{F}\big(\nu(n)\big)n^{-s}.$$
Then the Hardy space $\h_p(X)$ of $X$-valued Dirichlet series is defined as the image of $H_p(\TT,X)$ under the Bohr transform $\B$, endowed with the norm
$$\|f\|_{\h_p(X)}:=\left\|\B^{-1}(f)\right\|_{H_p(\TT,X)},\quad f\in\h_p(X).$$
This scale of Hardy spaces was introduced in \cite{Ba,HLS} for scalar-valued Dirichlet series, and in \cite{CDS14} for Dirichlet series with values in a Banach space. We refer to \cite{DGMS,QQ} for more information.

We will also consider some larger Hardy spaces of vector-valued Dirichlet series. Given $u\in\C$ and $f\in\D(X)$, let $f_u$ denote the translation of $f$ by $u$, i.e. $f_u(\cdot):=f(\cdot+u)$. For $1\leq p<\infty$, the Hardy space $\h^+_p(X)$, introduced in \cite{DP}, consists of $f\in\D(X)$ such that $f_{\sigma}\in\h_p(X)$ for any $\sigma>0$ and
$$\|f\|_{\h^+_p(X)}:=\sup_{\sigma>0}\|f_{\sigma}\|_{\h_p(X)}<\infty.$$
It was shown in \cite{DP} that $\h_p(X)$ is isometrically embedded into $\h^+_p(X)$, and $\h_p(X)=\h^+_p(X)$ if and only if $X$ has the analytic Radon--Nikod\'{y}m property. For $1\leq p<\infty$, let $\h^{\text{weak}}_p(X)$ be the weak version of Hardy space of Dirichlet series in $\D(X)$. More precisely, the space $\h^{\text{weak}}_p(X)$ consists of Dirichlet series $f\in\D(X)$ such that $x^*\circ f\in\h_p(\C)$ for every $x^*\in X^*$ and
$$\|f\|_{\h^{\text{weak}}_p(X)}:=\sup_{x^*\in B_{X^*}}\|x^*\circ f\|_{\h_p(\C)}<\infty,$$
where $B_{X^*}$ is the closed unit ball of $X^*$. It is clear that $\h^+_p(X)\subset\h^{\text{weak}}_p(X)$, and by \cite[Example 15]{La05}, $\h^+_p(X)\subsetneq\h^{\text{weak}}_p(X)$ and $\|\cdot\|_{\h^{\text{weak}}_p(X)}$ is not equivalent to $\|\cdot\|_{\h^+_p(X)}$ on $\h^+_p(X)$ if $X$ is infinite-dimensional (see also \cite{FGP,LaTy}).

Given $g\in\D(\C)$, the Volterra operator $T_g$ is defined for $f\in\D(X)$ by
$$T_gf(s):=-\int_s^{+\infty}f(u)g'(u)du,$$
where $\Re s$ is large enough. This operator was first introduced by Pommerenke \cite{Po} in the setting of analytic functions on the unit disk $\DD$ of $\C$, and the above Dirichlet series analogy was defined by Brevig, Perfekt and Seip in \cite{BPS}, where they gave some necessary and sufficient conditions for the boundedness of $T_g$ acting on the Hardy spaces $\h_p(\C)$. Motivated by this, and in view of the fact that $\h^+_p(X)$ and $\h^{\text{weak}}_p(X)$ are essentially different spaces for any infinite-dimensional Banach space $X$, we here investigate the Volterra operators $T_g$ that are bounded from $\h^{\text{weak}}_p(X)$ to $\h^+_p(X)$. This problem was initially considered by Laitila, Tylli and Wang \cite{LTW} for composition operators in the setting of Hardy and Bergman spaces of vector-valued analytic functions on $\DD$. Later on, Chen and Wang \cite{CW} characterized the Volterra operators that are bounded from weak to strong Hardy (and Bergman) spaces of vector-valued analytic functions on $\DD$.

To state our main result, we need the notion of uniform PL-convexity of a complex Banach space. For $1\leq p<\infty$, the modulus of PL-convexity $\delta^X_p(\epsilon)$ ($\epsilon>0$) of the space $X$ is defined by
$$\delta^X_p(\epsilon):=\inf\left\{\left(\int_{\T}\|x+\xi y\|^pdm(\xi)\right)^{1/p}-1:x,y\in X,\ \|x\|_X=1,\ \|y\|_X=\epsilon\right\}.$$
The space $X$ is said to be uniformly PL-convex if $\delta^X_1(\epsilon)>0$ for all $\epsilon>0$, and for $2\leq p<\infty$, $X$ is said to be $p$-uniformly PL-convex if there exists $C>0$ such that $\delta^X_p(\epsilon)\geq C\epsilon^p$ for all $\epsilon>0$. It is well-known (see \cite{DGT}) that $X$ is $p$-uniformly PL-convex if and only if there exists $C>0$ such that
$$\int_{\T}\|x+\xi y\|^p_Xdm(\xi)\geq\|x\|^p_X+C\|y\|^p_X,\quad \forall x,y\in X.$$

For $1\leq p<\infty$, let $H_p(\DD,X)$ be the Hardy space consisting of $X$-valued analytic functions $f$ on the unit disk $\DD$ such that 
$$\|f\|_{H_p(\DD,X)}:=\sup_{0<r<1}\left(\int_{\T}\|f(r\xi)\|^p_Xdm(\xi)\right)^{1/p}<\infty.$$
The corresponding weak version $H^{\text{weak}}_p(\DD,X)$ can be defined as before. It was shown in \cite{CW} that for $2\leq p<\infty$ and any infinite-dimensional $p$-uniformly PL-convex space $X$, the boundedness of Volterra operators from $H^{\text{weak}}_p(\DD,X)$ to $H_p(\DD,X)$ is related to the membership of Schatten $p$-class of Volterra operators on the Hardy space $H_2(\DD,\C)$. On the other hand, Brevig, Perfekt and Seip \cite[Theorem 7.2]{BPS} proved that there is no nontrivial Volterra operator $T_g$ in the Schatten class $S_p(\h_2(\C))$ for all $0<p<\infty$. Hence it is reasonable to expect that for $2\leq p<\infty$ and any infinite-dimensional $p$-uniformly PL-convex space $X$, there is no nontrivial bounded Volterra operator $T_g$ from $\h^{\text{weak}}_p(X)$ to $\h^+_p(X)$. Our main result establishes that this is the case.

\begin{theorem}\label{main}
Let $2\leq p<\infty$, $g\in\D(\C)$, and let $X$ be infinite-dimensional and $p$-uniformly PL-convex. If $T_g:\h_p^{\text{weak}}(X)\to\h_p^+(X)$ is bounded, then $g$ is constant.
\end{theorem}

In order to prove the above theorem, we need to estimate the norm of $f\in\h^+_p(X)$ from below via its derivative $f'$. A classical result of this style is the Littlewood--Paley inequality (see \cite{LP36} or \cite[Theorem 4.4.4]{JVA}), which indicates that if $2\leq p<\infty$, then there exists $C>0$ such that for any $f\in H_p(\DD,\C)$,
$$\|f\|_{H_p(\DD,\C)}\geq\left(|f(0)|^p+C\int_{\DD}|f'(\xi)|^p(1-|\xi|^2)^{p-1}dA(\xi)\right)^{1/p},$$
where $dA$ is the Lebesgue area measure on $\C$. Vector-valued versions of Littlewood--Paley theory have been considered by several authors for various reasons. In particular, Blasco and Pavlovi\'{c} \cite{BP} proved that for $2\leq p<\infty$, the Banach space $X$ is $p$-uniformly PL-convex if and only if there exists $C>0$ such that for every $f\in H_p(\DD,X)$,
\begin{equation}\label{LiPa}
	\|f\|_{H_p(\DD,X)}\geq\left(\|f(0)\|^p_X+C\int_{\DD}\|f'(\xi)\|^p_X(1-|\xi|^2)^{p-1}dA(\xi)\right)^{1/p}.
\end{equation}
Based on this inequality, we can establish the following Littlewood--Paley inequality for vector-valued Dirichlet series, which plays an essential role in the proof of Theorem \ref{main}.

\begin{theorem}\label{L-P}
Let $2\leq p<\infty$ and $X$ be a $p$-uniformly PL-convex space. Then there exists $C>0$ such that for any $f\in\h^+_p(X)$,
$$\|f(+\infty)\|_X+\left(\int_0^{+\infty}\|f'_{\sigma}\|^p_{\h_p(X)}\sigma^{p-1}d\sigma\right)^{1/p}\leq C\|f\|_{\h^+_p(X)}.$$
\end{theorem}

For any $\alpha>-1$ and $1\leq p<\infty$, we define the McCarthy--Dirichlet space $\mathscr{D}^p_{\alpha}(X)$ of $X$-valued Dirichlet series as the completion of $\mathcal{P}(X)$ with respect to the norm
$$\|P\|_{\mathscr{D}^p_{\alpha}(X)}:=\|P(+\infty)\|_X+\left(\int_0^{+\infty}\|P'_{\sigma}\|^p_{\h_p(X)}\sigma^{\alpha}d\sigma\right)^{1/p},
\quad P\in\mathcal{P}(X).$$
Then Theorem \ref{L-P} can be restated as follows: if $2\leq p<\infty$ and $X$ is a $p$-uniformly PL-convex Banach space, then we have the bounded inclusion
$$\h^+_p(X)\subset\mathscr{D}^p_{p-1}(X).$$
This can be compared with the inclusion between classical Hardy and Dirichlet spaces over the unit disk.

Theorems \ref{main} and \ref{L-P} are proven in Section \ref{proofs}. We also give some remarks regarding the Littlewood--Paley inequalities of Dirichlet series in Section \ref{remark}.

Throughout the paper, the letter $C$ always denotes a positive constant whose value is not essential and may change from one occurrence to the next. We also write $A\lesssim B$ or $B\gtrsim A$ if $A\leq CB$ for some inessential constant $C>0$. For a Dirichlet series $f(s)=\sum_{n=1}^{\infty}x_nn^{-s}$, we always use $f(+\infty)$ to denote $x_1$.

\section{Proofs of Theorems \ref{main} and \ref{L-P}}\label{proofs}

In this section, we are going to prove Theorems \ref{main} and \ref{L-P}. Before proceeding, we introduce some auxiliary results.

We first explain more about the definition of the Hardy spaces $\h_p(X)$. Due to the definition, a Dirichlet series $f(s)=\sum_{n=1}^{\infty}x_nn^{-s}$ belongs to $\h_p(X)$ if and only if there exists $F\in H_p(\TT,X)$ such that $\widehat{F}\big(\nu(n)\big)=x_n$ for every $n\in\mathbb{N}$. In this case, one has $\|f\|_{\h_p(X)}=\|F\|_{H_p(\TT,X)}$. In particular, for any Dirichlet polynomial $P(s)=\sum_{n=1}^Nx_nn^{-s}$, $\B^{-1}(P)(z)=\sum_{n=1}^Nx_nz^{\nu(n)}$, and
\begin{equation}\label{polynomial}
\|P\|_{\h_p(X)}=\left(\int_{\TT}\left\|\sum_{n=1}^Nx_nz^{\nu(n)}\right\|^p_Xdm_{\infty}(z)\right)^{1/p}.
\end{equation}
It is clear that for $1\leq p<\infty$ the Bohr transform $\B$ is an isometric isomorphism from $H_p(\TT,X)$ onto $\h_p(X)$. There are two elementary consequences of this fact.
\begin{enumerate}[(i)]
	\item The coefficients of a Dirichlet series $f\in\h_p(X)$ are bounded by $\|f\|_{\h_p(X)}$. Consequently, if we use $c_n(f)$ to denote the $n$th Dirichlet coefficient of $f\in\D(X)$, then the convergence $f_j\to f$ in $\h_p(X)$ implies the convergence $c_n(f_j)\to c_n(f)$ for every $n\in\mathbb{N}$.
	\item The set $\mathcal{P}(X)$ of Dirichlet polynomials is dense in $\h_p(X)$ for $1\leq p<\infty$ (see \cite[Proposition 24.6]{DGMS}).
\end{enumerate}

The following lemma concerns the horizontal translation of Dirichlet series in Hardy spaces, which can be found in \cite[Proposition 2.3]{DP}.

\begin{lemma}\label{decr}
	Suppose $1\leq p<\infty$ and $f\in\h_p(X)$. Then for any $\sigma>0$, $f_{\sigma}\in\h_p(X)$. Moreover, the function $\sigma\mapsto\|f_{\sigma}\|_{\h_p(X)}$ is decreasing on $[0,\infty)$.
\end{lemma}

For any $P(s)=\sum_{n=1}^Nx_nn^{-s}$ in $\mathcal{P}(X)$ and $w\in\TT$, write
$$P_w(s):=\sum_{n=1}^Na_nw^{\nu(n)}n^{-s}.$$
The following lemma is an immediate consequence of the rotation invariance of the measure $m_{\infty}$ and \eqref{polynomial}.

\begin{lemma}\label{ww}
	Let $1\leq p<\infty$ and $P\in\mathcal{P}(X)$. Then for any $s=\sigma+it\in\C$,
	$$\int_{\TT}\|P_w(s)\|^p_Xdm_{\infty}(w)=\|P_{\sigma}\|^p_{\h_p(X)}.$$
\end{lemma}

Given $N\in\mathbb{N}$, let $S_N$ be the partial sum operator defined by
$$S_N\left(\sum_{n=1}^{\infty}x_nn^{-s}\right):=\sum_{n=1}^{N}x_nn^{-s}.$$
The estimates of the operators $S_N$ are crucial for the modern theory of Dirichlet series. It was proved in \cite[Theorem 3.2]{DP} that there exists $C>0$ such that for every $N\in\mathbb{N}$ and every $1\leq p<\infty$,
\begin{equation}\label{partial}
\|S_N\|_{\h^+_p(X)\to\h^+_p(X)}\leq C\log N.
\end{equation}
Based on the above estimate, we can establish the following proposition on derivatives of Dirichlet series in Hardy spaces.

\begin{proposition}\label{deri}
Let $1\leq p<\infty$. Then for any $\sigma>0$ and $f\in\h^+_p(X)$, $f'_{\sigma}\in\h_p(X)$.
\end{proposition}
\begin{proof}
Fix $\sigma>0$, and suppose that $f(s)=\sum_{n=1}^{\infty}x_nn^{-s}$ belongs to $\h^+_p(X)$. Then for any $2\leq N<M$, Abel's summation formula yields that
\begin{align*}
\sum_{n=N}^{M}x_nn^{-(s+\frac{\sigma}{2})}\log n
=&\sum_{n=N}^{M-1}\left(\sum_{k=1}^nx_kk^{-s}\right)\left(n^{-\sigma}\log n-(n+1)^{-\sigma}\log(n+1)\right)\\
&+\left(\sum_{k=1}^Mx_kk^{-s}\right)M^{-\sigma}\log M-\left(\sum_{k=1}^{N-1}x_kk^{-s}\right)N^{-\sigma}\log N.
\end{align*}
Taking norms and using \eqref{partial}, we obtain that
\begin{align*}
&\left\|\sum_{n=N}^{M}x_nn^{-(s+\frac{\sigma}{2})}\log n\right\|_{\h^+_p(X)}\\
&\ \lesssim\sum_{n=N}^{M-1}\log n\left|n^{-\sigma}\log n-(n+1)^{-\sigma}\log(n+1)\right|+M^{-\sigma}\log^2M+N^{-\sigma}\log^2N\\
&\ \lesssim\sum_{n=N}^{M-1}n^{-\sigma-1}\log^2n+M^{-\sigma}\log^2M+N^{-\sigma}\log^2N\to0
\end{align*}
as $N,M\to\infty$. Therefore, there exists $g\in\h^+_p(X)$ such that
$$\sum_{n=2}^{N}x_nn^{-(\cdot+\frac{\sigma}{2})}\log n\to g(\cdot)$$
in $\h^+_p(X)$ as $N\to\infty$. Consequently,
$$\sum_{n=2}^{N}x_nn^{-(\cdot+\sigma)}\log n\to g_{\sigma/2}(\cdot)$$
in $\h_p(X)$ as $N\to\infty$. Comparing the Dirichlet coefficients, we finally conclude that $f'_{\sigma}=-g_{\sigma/2}\in\h_p(X)$.
\end{proof}

We are now ready to prove Theorem \ref{L-P}.

\begin{proof}[Proof of Theorem \ref{L-P}]
For $\eta>0$, let $\phi_{\eta}:\DD\to\C_0$ be the Cayley transform defined by
$$\phi_{\eta}(\xi)=\eta\frac{1+\xi}{1-\xi},\quad \xi\in\DD,$$
where $\C_0:=\{s\in\C:\Re s>0\}$. Then for any $P\in\mathcal{P}(X)$ and $w\in\TT$, applying \eqref{LiPa} to the function $P_w\circ\phi_{\eta}$, we obtain that
$$\|P_w(\eta)\|^p_X+C\int_{\DD}\|(P_w\circ\phi_{\eta})'(\xi)\|_X^p(1-|\xi|^2)^{p-1}dA(\xi)\leq\|P_w\circ\phi_{\eta}\|^p_{H_p(\DD,X)}.$$
Using the change of variables $s=\sigma+it=\phi_{\eta}(\xi)$, we get
\begin{align*}
&\int_{\DD}\|(P_w\circ\phi_{\eta})'(\xi)\|_X^p(1-|\xi|^2)^{p-1}dA(\xi)\\
&\ \ =\int_{\DD}\|P'_w(\phi_{\eta}(\xi))\|^p_X|\phi'_{\eta}(\xi)|^p(1-|\xi|^2)^{p-1}dA(\xi)\\
&\ \ =\int_{\C_0}\|P'_w(s)\|^p_X|\phi'_{\eta}(\phi^{-1}_{\eta}(s))|^{p-2}(1-|\phi^{-1}_{\eta}(s)|^2)^{p-1}dA(s)\\
&\ \ =\int_{\C_0}\|P'_w(s)\|^p_X\left(\frac{2\eta}{\left|1-\frac{s-\eta}{s+\eta}\right|^2}\right)^{p-2}
  \left(1-\left|\frac{s-\eta}{s+\eta}\right|^2\right)^{p-1}dA(s)\\
&\ \ =2^p\int_0^{+\infty}\int_{\R}\|P'_w(\sigma+it)\|^p_X\frac{\sigma^{p-1}\eta}{(\sigma+\eta)^2+t^2}dtd\sigma.
\end{align*}
Therefore, we establish that
$$\|P_w(\eta)\|^p_X+C\int_0^{+\infty}\int_{\R}\|P'_w(\sigma+it)\|^p_X\frac{\sigma^{p-1}\eta}{(\sigma+\eta)^2+t^2}dtd\sigma
\leq\|P_w\circ\phi_{\eta}\|^p_{H_p(\DD,X)}.$$
Integrating with respect to $w$ on $\TT$, and using Fubini's theorem, Fatou's lemma, Lemmas \ref{ww} and \ref{decr}, we arrive at
\begin{align*}
&\|P_{\eta}\|^p_{\h_p(X)}+C\int_0^{+\infty}\int_{\R}\|P'_{\sigma}\|^p_{\h_p(X)}\frac{\sigma^{p-1}\eta}{(\sigma+\eta)^2+t^2}dtd\sigma\\
&\ \ =\int_{\TT}\left(\|P_w(\eta)\|^p_X+
  C\int_0^{+\infty}\int_{\R}\|P'_w(\sigma+it)\|^p_X\frac{\sigma^{p-1}\eta}{(\sigma+\eta)^2+t^2}dtd\sigma\right)dm_{\infty}(w)\\
&\ \ \leq\int_{\TT}\lim_{r\to\infty}\int_{\T}\left\|P_w(\phi_{\eta}(r\zeta))\right\|^p_Xdm(\zeta)dm_{\infty}(w)\\
&\ \ \leq\liminf_{r\to1}\int_{\T}\int_{\TT}\|P_w(\phi_{\eta}(r\zeta))\|^p_{X}dm_{\infty}(w)dm(\zeta)\\
&\ \ =\liminf_{r\to1}\int_{\T}\Big\|P\Big(\cdot+\Re(\phi_{\eta}(r\zeta))\Big)\Big\|^p_{\h_p(X)}dm(\zeta)\\
&\ \ \leq\|P\|^p_{\h_p(X)},
\end{align*}
which, together with the fact that $\frac{\sigma+\eta}{\pi((\sigma+\eta)^2+t^2)}dt$ is a probability measure on $\R$, implies that
$$\|P_{\eta}\|^p_{\h_p(X)}+C\int_0^{+\infty}\|P'_{\sigma}\|^p_{\h_p(X)}\frac{\sigma^{p-1}\eta}{\sigma+\eta}d\sigma\leq\|P\|^p_{\h_p(X)}.$$
Letting $\eta\to+\infty$ and using Lebesgue's dominated convergence theorem, we conclude that
$$\|P(+\infty)\|^p_{X}+C\int_0^{+\infty}\|P'_{\sigma}\|^p_{\h_p(X)}\sigma^{p-1}d\sigma\leq\|P\|^p_{\h_p(X)}.$$
Since Dirichlet polynomials are dense in $\h_p(X)$, it follows that there exists $C>0$ such that for any $g\in\h_p(X)$,
\begin{equation}\label{lpp}
\|g(+\infty)\|^p_X+\int_0^{+\infty}\|g'_{\sigma}\|^p_{\h_p(X)}\sigma^{p-1}d\sigma\leq C\|g\|^p_{\h_p(X)}.
\end{equation}
Suppose now that $f\in\h^+_p(X)$. Then for any $\delta>0$, $f_{\delta}\in\h_p(X)$ and by Proposition \ref{deri}, $f'_{\delta}\in\h_p(X)$. Applying \eqref{lpp} to the Dirichlet series $f_{\delta}$ yields that
$$\|f(+\infty)\|^p_X+\int_0^{+\infty}\|f'_{\delta+\sigma}\|^p_{\h_p(X)}\sigma^{p-1}d\sigma
\leq C\|f_{\delta}\|^p_{\h_p(X)}\leq C\|f\|^p_{\h^+_p(X)}.$$
In view of Lemma \ref{decr}, we may let $\delta\to0$ and use Lebesgue's monotone convergence theorem to conclude that
$$\|f(+\infty)\|^p_X+\int_0^{+\infty}\|f'_{\sigma}\|^p_{\h_p(X)}\sigma^{p-1}d\sigma\leq C\|f\|^p_{\h^+_p(X)},$$
which finishes the proof.
\end{proof}

To establish Theorem \ref{main}, we need some more auxiliary results. The following Dvoretzky theorem can be found in \cite[Chapter 19]{DJT}.

\begin{theorem}\label{Dvo}
For any $N\in\N$ and $\epsilon>0$ there is $c(N,\epsilon)\in\N$ so that for any Banach space $X$ of dimension at least $c(N,\epsilon)$, there is a linear embedding $E_N:l^N_2\to X$ so that
\begin{equation}\label{E_n}
(1+\epsilon)^{-1}\biggl(\sum_{n=1}^N|a_n|^2\biggr)^{1/2}\leq\biggl\|\sum_{n=1}^Na_nE_Ne_n\biggr\|_X
\leq\biggl(\sum_{n=1}^N|a_n|^2\biggr)^{1/2}
\end{equation}
for any $a_1,\dots,a_N\in\C$. Here $(e_1,\dots,e_N)$ is some fixed orthonormal basis of $l^N_2$.
\end{theorem}

Recall that a sequence $\{\lambda_n\}_{n\geq1}$ of complex numbers is completely multiplicative if $\lambda_{mn}=\lambda_m\lambda_n$ for all $m,n\in\mathbb{N}$. For $1\leq p,q<\infty$, a sequence $\{\lambda_n\}_{n\geq1}$ is said to be a multiplier from $\h_p(\C)$ to $\h_q(\C)$ if $\sum_{n=1}^{\infty}\lambda_na_nn^{-s}\in\h_q(\C)$ for each $f(s)=\sum_{n=1}^{\infty}a_nn^{-s}$ in $\h_p(\C)$. Note that by the closed graph theorem, if $\{\lambda_n\}_{n\geq1}$ is a multiplier from $\h_p(\C)$ to $\h_q(\C)$, then there exists $C>0$ such that for any $f(s)=\sum_{n=1}^{\infty}a_nn^{-s}$ in $\h_p(\C)$,
$$\left\|\sum_{n=1}^{\infty}\lambda_na_nn^{-s}\right\|_{\h_q(\C)}\leq C\left\|\sum_{n=1}^{\infty}a_nn^{-s}\right\|_{\h_p(\C)}.$$
The following lemma is due to Bayart \cite{Ba02T}.

\begin{lemma}\label{multi}
Let $1\leq p\leq q<\infty$ and $\{\lambda_n\}_{n\geq1}$ be a completely multiplicative sequence such that $\lambda_{\mathfrak{p}_j}\leq\sqrt{p/q}$ for large $j$. Then $\{\lambda_n\}_{n\geq1}$ is a multiplier from $\h_p(\C)$ to $\h_q(\C)$.
\end{lemma}

We also need the following lemma.

\begin{lemma}\label{B-times}
Let $1\leq p<\infty$. Suppose that $f(s)=\sum_{n=1}^Nx_nn^{-s}$ belongs to $\mathcal{P}(X)$, and $g(s)=\sum_{n=1}^{\infty}b_nn^{-s}$ belongs to $\h_p(\C)$. Then $fg\in\h_p(X)$ and $\B^{-1}(fg)=\B^{-1}(f)\B^{-1}(g)$.
\end{lemma}
\begin{proof}
Write $F=\B^{-1}(f)$ and $G=\B^{-1}(g)$. Then it is clear that $FG\in H_p(\TT,X)$. Moreover,
$$F(z)=\sum_{n=1}^Nx_nz^{\nu(n)},\quad z\in\TT,$$
and for any $n\in\mathbb{N}$, $\widehat{G}\big(\nu(n)\big)=b_n$. We now calculate the $\nu$th Fourier coefficient of $FG$ for $\nu\in\mathbb{N}_0^{(\infty)}$. Bearing in mind that $G\in H_p(\TT,X)$, we have
\begin{align*}
\widehat{FG}(\nu)&=\int_{\TT}(FG)(z)z^{-\nu}dm_{\infty}(z)\\
&=\int_{\TT}F(z)G(z)z^{-\nu}dm_{\infty}(z)\\
&=\sum_{n=1}^{N}x_n\int_{\TT}G(z)z^{-(\nu-\nu(n))}dm_{\infty}(z)\\
&=\sum_{\substack{n|\mathfrak{p}^{\nu}\\1\leq n\leq N}}x_n\widehat{G}\big(\nu-\nu(n)\big).
\end{align*}
On the other hand, it is easy to see that the $l$th Dirichlet coefficient $c_l(fg)$ of $fg$ is given by
$$c_l(fg)=\sum_{\substack{n|l\\1\leq n\leq N}}x_nb_{l/n}.$$
Therefore, $\widehat{FG}\big(\nu(l)\big)=c_l(fg)$ for any $l\in\mathbb{N}$, which implies that $\B(FG)=fg$ and $fg\in\h_p(X)$.
\end{proof}

We are now in a position to prove Theorem \ref{main}.

\begin{proof}[Proof of Theorem \ref{main}]
Suppose that $T_g:\h_p^{\text{weak}}(X)\to\h_p^+(X)$ is bounded. Then it is easy to see that $g\in\h_p(\C)$. Consequently, by Proposition \ref{deri}, $g'_{\sigma}\in\h_p(\C)$ for any $\sigma>0$. We will complete the proof by showing $\|g'_{1/2}\|_{\h_p(\C)}=0$.

Let $N\in\N$ and $\epsilon>0$. According to Theorem \ref{Dvo}, we may fix a linear embedding $E_N:l^N_2\to X$ such that (\ref{E_n}) holds. Put $x^{(N)}_n=E_Ne_n$ for $n=1,2,\dots,N$, and let $\lambda_n=\left(2p^{-1}\right)^{\frac{\Omega(n)}{2}}$ for $n\geq1$, where $\Omega(n)$ is the number of prime factors of $n$, counted with multiplicity. Define the $X$-valued Dirichlet polynomial $f_N$ by
$$f_N(s)=\sum_{n=1}^N\lambda_nx^{(N)}_nn^{-s}=E_N\left(\sum_{n=1}^N\lambda_ne_nn^{-s}\right).$$
Then by Lemma \ref{multi},
\begin{align*}
\|f\|_{\h^{\text{weak}}_p(X)}&=\sup_{x^*\in B_{X^*}}\|x^*\circ f_N\|_{\h_p(\C)}\\
&=\sup_{x^*\in B_{X^*}}\left\|\sum_{n=1}^N\lambda_nx^*\left(x^{(N)}_n\right)n^{-s}\right\|_{\h_p(\C)}\\
&\lesssim\sup_{x^*\in B_{X^*}}\left\|\sum_{n=1}^Nx^*\left(x^{(N)}_n\right)n^{-s}\right\|_{\h_2(\C)}\\
&=\sup_{x^*\in B_{X^*}}\left(\sum_{n=1}^N\left|E^*_Nx^*(e_n)\right|^2\right)^{1/2}\leq1.
\end{align*}
Hence it follows from Theorem \ref{L-P} and Lemma \ref{B-times} that
\begin{align*}
\|T_g\|^p
&\gtrsim\|T_gf_N\|^p_{\h^+_p(X)}\\
&\gtrsim\int_0^{+\infty}\|(f_N)_{\sigma}g'_{\sigma}\|^p_{\h_p(X)}\sigma^{p-1}d\sigma\\
&=\int_0^{+\infty}\left\|\B^{-1}\Big((f_N)_{\sigma}g'_{\sigma}\Big)\right\|^p_{H_p(\TT,X)}\sigma^{p-1}d\sigma\\
&=\int_0^{+\infty}\left\|\B^{-1}\Big((f_N)_{\sigma}\Big)\B^{-1}(g'_{\sigma})\right\|^p_{H_p(\TT,X)}\sigma^{p-1}d\sigma\\
&=\int_0^{+\infty}\sigma^{p-1}\int_{\TT}\left\|\B^{-1}\Big((f_N)_{\sigma}\Big)(z)\right\|^p_X|\B^{-1}(g'_{\sigma})(z)|^pdm_{\infty}(z)d\sigma.
\end{align*}
Note that for any $z\in\TT$, it follows from \eqref{E_n} that
\begin{align*}
\left\|\B^{-1}\Big((f_N)_{\sigma}\Big)(z)\right\|_X&=\left\|\sum_{n=1}^{N}\lambda_nn^{-\sigma}z^{\nu(n)}E_Ne_n\right\|_X\\
&\gtrsim\left(\sum_{n=1}^N(2p^{-1})^{\Omega(n)}n^{-2\sigma}\right)^{1/2}\\
&\gtrsim\left(\sum_{1\leq j\leq\pi(N)}\mathfrak{p}_j^{-2\sigma}\right)^{1/2},
\end{align*}
where $\pi(N)$ denotes the number of primes less than or equal to $N$. Therefore, we may apply Lemma \ref{decr} to obtain that
\begin{align*}
\|T_g\|^p
&\gtrsim\int_0^{+\infty}\left(\sum_{1\leq j\leq\pi(N)}\mathfrak{p}_j^{-2\sigma}\right)^{p/2}\sigma^{p-1}
  \int_{\TT}|\B^{-1}(g'_{\sigma})(z)|^pdm_{\infty}(z)d\sigma\\
&\geq\int_0^{1/2}\|g'_{\sigma}\|^p_{\h_p(\C)}\left(\sum_{1\leq j\leq\pi(N)}\mathfrak{p}_j^{-2\sigma}\right)^{p/2}\sigma^{p-1}d\sigma\\
&\gtrsim\|g'_{1/2}\|^p_{\h_p(\C)}\left(\sum_{1\leq j\leq\pi(N)}\mathfrak{p}_j^{-1}\right)^{p/2}.
\end{align*}
Since $N\in\mathbb{N}$ is arbitrary and $\sum_{j=1}^{\infty}\mathfrak{p}_j^{-1}=\infty$, we conclude that $\|g'_{1/2}\|_{\h_p(\C)}=0$, which finishes the proof.
\end{proof}

\section{Concluding remarks}\label{remark}

In this section, we give some remarks regarding the Littlewood--Paley inequalities of Dirichlet series. We begin from two generalizations of Theorem \ref{L-P}. The following corollary concerns $f\in\h^+_q(X)$ with $p$-uniformly PL-convex space $X$, where $p\geq\max\{2,q\}$.

\begin{corollary}
Let $2\leq p<\infty$, $1\leq q\leq p$, and let $X$ be a $p$-uniformly PL-convex space. Then there exists $C>0$ such that for any $f\in\h^+_q(X)$,
$$\|f(+\infty)\|_X+\left(\int_0^{+\infty}\|f'_{\sigma}\|^p_{\h_q(X)}\sigma^{p-1}d\sigma\right)^{1/p}\leq C\|f\|_{\h^+_q(X)}.$$
\end{corollary}
\begin{proof}
Let $P(s)=\sum_{n=1}^{N}x_nn^{-s}$ belong to $\mathcal{P}(X)$. Define
$$h(u)(s):=P_u(s)=\sum_{n=1}^{N}x_nn^{-s}n^{-u},\quad u\in\C_0.$$
Then by \cite[Theorem 2.1]{DP}, \eqref{polynomial} and the rotation invariance of $m_{\infty}$,
\begin{align*}
\|h\|^p_{\h^+_p(\h_q(X))}
&=\|h\|^p_{\h_p(\h_q(X))}\\
&=\int_{\TT}\left\|\sum_{n=1}^{N}(x_nn^{-s})z^{\nu(n)}\right\|^p_{\h_q(X)}dm_{\infty}(z)\\
&=\int_{\TT}\left(\int_{\TT}\left\|\sum_{n=1}^{N}x_nz^{\nu(n)}w^{\nu(n)}\right\|^q_X
  dm_{\infty}(w)\right)^{p/q}dm_{\infty}(z)\\
&=\left(\int_{\TT}\left\|\sum_{n=1}^{N}x_nw^{\nu(n)}\right\|^q_Xdm_{\infty}(w)\right)^{p/q}\\
&=\|P\|^p_{\h_q(X)}.
\end{align*}
It follows from \cite[Theorem 4.1]{DGT} that $L_q(\TT,X)$ is $p$-uniformly PL-convex, which implies that $\h_q(X)$ is $p$-uniformly PL-convex. Therefore, we may apply Theorem \ref{L-P} to obtain that
\begin{align*}
\|h(+\infty)\|_{\h_q(X)}+\left(\int_0^{+\infty}\|h'_{\sigma}\|^p_{\h_p(\h_q(X))}\sigma^{p-1}d\sigma\right)^{1/p}\\
\leq C\|h\|_{\h^+_p(\h_q(X))}=C\|P\|_{\h_q(X)}.
\end{align*}
It is clear that $\|h(+\infty)\|_{\h_q(X)}=\|f(+\infty)\|_{X}$. Moreover, the rotation invariance gives that $\|h'_{\sigma}\|^p_{\h_p(\h_q(X))}=\|P'_{\sigma}\|^p_{\h_q(X)}$ for any $\sigma>0$ as before. Consequently,
$$\|P(+\infty)\|_{X}+\left(\int_0^{+\infty}\|P'_{\sigma}\|^p_{\h_q(X)}\sigma^{p-1}d\sigma\right)^{1/p}\leq C\|P\|_{\h_q(X)}.$$
Arguing as in the proof of Theorem \ref{L-P}, we can establish the desired result.
\end{proof}

The following theorem concerns $f\in\h^+_q(X)$ with $p$-uniformly PL-convex space $X$, where $2\leq p\leq q$.

\begin{theorem}
Let $2\leq p\leq q<\infty$, and let $X$ be $p$-uniformly PL-convex. Then there exists $C>0$ such that for any $f\in\h^+_q(X)$,
\begin{align*}
\|f(+\infty)\|_X+\left(\int_0^{+\infty}\int_{\TT}\left\|\B^{-1}(f'_{\sigma})(z)\right\|^p_X
\left\|\B^{-1}(f_{\sigma})(z)\right\|^{q-p}_Xdm_{\infty}(z)\sigma^{p-1}d\sigma\right)^{1/q}\\
\leq C\|f\|_{\h^+_q(X)}.
\end{align*}
\end{theorem}
\begin{proof}
If $f\in\h^+_q(X)$, then for any $\sigma>0$, $f_{\sigma}\in\h_q(X)$ and by Proposition \ref{deri}, $f'_{\sigma}\in\h_q(X)$. Hence both $\B^{-1}(f_{\sigma})$ and $\B^{-1}(f'_{\sigma})$ are well-defined.

Using \cite[Theorem 2.6]{BP} instead of \eqref{LiPa}, and arguing as in the proof of Theorem \ref{L-P}, we obtain that for any $g\in\h_q(X)$,
\begin{align}\label{ggg}
\nonumber\|g(+\infty)\|^q_X+\int_0^{+\infty}\sigma^{p-1}\int_{\TT}\left\|\B^{-1}(g'_{\sigma})\right\|^p_X
\left\|\B^{-1}(g_{\sigma})\right\|^{q-p}_Xdm_{\infty}d\sigma\\
\leq C\|g\|^q_{\h_q(X)}.
\end{align}
Suppose now $f\in\h^+_q(X)$. Then for any $\delta>0$, applying \eqref{ggg} to $f_{\delta}$ yields that
\begin{align*}
\|f(+\infty)\|^q_X+\int_0^{+\infty}\sigma^{p-1}\int_{\TT}\left\|\B^{-1}(f'_{\sigma+\delta})\right\|^p_X
\left\|\B^{-1}(f_{\sigma+\delta})\right\|^{q-p}_Xdm_{\infty}d\sigma\\
\leq C\|f\|^q_{\h^+_q(X)}.
\end{align*}
Equivalently,
\begin{align*}
\|f(+\infty)\|^q_X+\int_0^{+\infty}\chi_{\delta}(\sigma)(\sigma-\delta)^{p-1}\int_{\TT}\left\|\B^{-1}(f'_{\sigma})\right\|^p_X
\left\|\B^{-1}(f_{\sigma})\right\|^{q-p}_Xdm_{\infty}d\sigma\\
\leq C\|f\|^q_{\h^+_q(X)},
\end{align*}
where $\chi_{\delta}$ is the characteristic function of $[\delta,+\infty)$. Since the above inequality holds for any $\delta>0$, letting $\delta\to0$ and using Fatou's lemma, we conclude the desired result.
\end{proof}

We now turn to the case $1\leq p\leq2$. In this case, the following Littlewood--Paley inequality for scalar-valued analytic functions on the unit disk $\DD$ is well-known (see \cite{LP36} or \cite[Theorem 4.4.4]{JVA}):
\begin{equation}\label{LP<2}
\|f\|_{H_p(\DD,\C)}\leq|f(0)|+C\left(\int_{\DD}|f'(\xi)|^p(1-|\xi|^2)^{p-1}dA(\xi)\right)^{1/p}.
\end{equation}
This inequality should be understood as follows: if the integral at the right-hand side is finite, then $f\in H_p(\DD,\C)$ and the norm of $f$ is less than or equal to the quantity at the right-hand side; but if the integral is infinite, then nothing can be said about $f$. Based on \eqref{LP<2}, we can use the same method as in the proof of Theorem \ref{L-P} to establish the following Littlewood--Paley inequality for scalar-valued Dirichlet series.

\begin{theorem}
Let $1\leq p\leq2$ and $f\in\D(\C)$. If for any $\sigma>0$, $f'_{\sigma}\in\h_p(\C)$, and
$$\int_0^{+\infty}\|f'_{\sigma}\|^p_{\h_p(\C)}\sigma^{p-1}d\sigma<\infty,$$
then $f\in\h_p(\C)$. Moreover, there exists some constant $C>0$, independent of $f$, such that
$$\|f\|_{\h_p(\C)}\leq |f(+\infty)|+C\left(\int_0^{+\infty}\|f'_{\sigma}\|^p_{\h_p(\C)}\sigma^{p-1}d\sigma\right)^{1/p}.$$
In other words, if $1\leq p\leq2$, then $\mathscr{D}^p_{p-1}(\C)\subset\h_p(\C)$, and the inclusion is bounded.
\end{theorem}

In order to establish the vector-valued version of the above theorem, one need to find the vector-valued version of \eqref{LP<2}. For harmonic functions on $\DD$ with values in a real, $p$-uniformly smooth Banach space, this was done in \cite{ADP}. Therefore, it is reasonable to guess that, for analytic functions on $\DD$ with values in a complex Banach space, the inequality of the form \eqref{LP<2} is related to some complex smoothness. However, to the best of our knowledge, we cannot find any references devoted to the notion of complex smoothness of Banach spaces. Here we are going to make some elementary attempts on this direction. Motivated by the modulus of smoothness of a real Banach space (see \cite[Definition 1.e.1]{LT}), we define the modulus of PL-smoothness $\rho_1^X(\tau)$ ($\tau>0$) of a complex Banach space $X$ as follows:
$$\rho_1^X(\tau):=\sup\left\{\int_{\T}\|x+\xi y\|_Xdm(\xi)-1:x,y\in X,\ \|x\|_X=1,\ \|y\|_X=\tau\right\}.$$
The Banach space $X$ is said to be uniformly PL-smooth if $\lim_{\tau\to0}\rho^X_1(\tau)/\tau=0$, and for $1<p\leq2$, $X$ is said to be $p$-uniformly PL-smooth if there exists $C>0$ such that $\rho^X_1(\tau)\leq C\tau^p$ for $\tau>0$. For any $1<p\leq 2$, it is not difficult to see that if there exists $C>0$ such that for any $X$-valued analytic function $f$ on $\DD$,
$$\|f\|_{H_p(\DD,X)}\leq\left(\|f(0)\|^p_X+C\int_{\DD}\|f'(\xi)\|^p_{X}(1-|\xi|^2)^{p-1}dA(\xi)\right)^{1/p},$$
then $X$ is $p$-uniformly PL-smooth. In order to establish the reverse implication, one need to consider the dual relation between uniform PL-smoothness and uniform PL-convexity (see \cite[Proposition 1.e.2]{LT} for the dual relation in real Banach spaces). This is of independent interest and is worthy to be investigated further.


\end{document}